\documentclass[12pt,english]{smfart}

\usepackage[T1]{fontenc}
\usepackage[english,french]{babel}

\usepackage{amssymb,url,xspace,smfthm}
\usepackage{csquotes}
\usepackage{stmaryrd}
\usepackage[all]{xy}
\usepackage{enumitem}
\usepackage{hyperref}

\usepackage{tikz}

\newtheorem{theorem}{Theorem}[section]
\newtheorem{proposition}[theorem]{Proposition}
\theoremstyle{remark}
  \newtheorem{remark}[theorem]{Remark}
  
\theoremstyle{definition}
  \newtheorem{definition}[theorem]{Definition}

\theoremstyle{definition}
    \newtheorem{example}[theorem]{Example}
\theoremstyle{definition}

\theoremstyle{definition}
    \newtheorem{lemma}[theorem]{Lemma}

\newcommand{\LL}{\mathbb{L}}

\newcommand{\R}{\mathbb{R}}
\newcommand{\PP}{\mathbb{P}}
\newcommand{\C}{\mathbb{C}}
\newcommand{\Q}{\mathbb{Q}}
\newcommand{\N}{\mathbb{N}}
\newcommand{\Z}{\mathbb{Z}}
\newcommand{\ac}{\overline{ac}}
\newcommand{\nash}{\hat{X}}
\newcommand{\field}{\mathbb{C}(\!(t)\!)}
\newcommand{\sheaf}{\mathcal{O}_{\hat{X}}(1)}
\newcommand{\tensor}{\mathcal{O}_Y(-\hat{K}_{Y/X}) \otimes \Omega_Y^2}

\title{Motivic local density of isolated surface singularities}
\author{Sidonie Ratajczak}
\address{Univ. Lille, CNRS, UMR 8524 - Laboratoire Paul Painlevé, F-59000 Lille, France}

\begin{document}

\begin{abstract}
The goal of this paper is to compute the motivic local density of an isolated algebraic surface singularity, in order to explain its link with algebraic multiplicity. In this context, we can use an additional data: the inner rate related to the bilipschitz geometry of the singularity, as studied by A. Belotto da Silva, L. Fantini and A. Pichon.
\end{abstract}

\maketitle

\tableofcontents

\section{Introduction}

Thanks to the pioneering work of Kontsevich, and then Denef and Loeser on motivic integration, non-archimedean geometry provides a framework to define and compare additive invariants associated to singularities of algebraic varieties. Denef and Loeser in \cite{denef_germs_1999} constructed a motivic zeta function, which interpolates the local $p$-adic Igusa functions, and that specializes on the cohomological invariants of the topological Milnor fiber. Another invariant is the local density, which exists in various contexts that we will recall. The goal of this paper is to compute the motivic local density of an isolated algebraic surface singularity. We consider $K$ a field equipped with a distance and a measure $\mu$, the local density of a set $X \subset K^n$ of dimension $d$ at a point $x$ is defined as the following limit, if it exists: 
$$
\Theta_d(X,x) = \lim_{r \longrightarrow 0} \frac{\mu_d(X \cap B_n(x,r))}{\mu_d(B_d(x,r))}.
$$
If $X$ is a smooth real or complex variety, the local density exists and is equal to 1 in each point. This notion becomes interesting when $x$ is a singular point, to which it attributes an invariant. 

It had been introduced in the complex case by P. Lelong in \cite{lelong_integration_1957}, where he shows the existence of local density of complex analytic sets. P. Thie shows later \cite{thie_lelong_1967} that in the complex analytic case, the local density is a positive integer, by expressing the local density as a sum of local densities of components of the tangent cone counted with multiplicites, and R. Draper \cite{draper_intersection_1969} expresses it as the algebraic multiplicity of the local ring of $X$ in $x$. K. Kurdyka and G. Raby start the study of the subanalytic real case in \cite{kurdyka_densite_1989}, showing the existence of the local density and expressing it as well as the sum of densities of components of the tangent cone counted with multiplicites. The density is no longer an integer, but a positive real number if $x \in \overline{X}$.

\vspace{\baselineskip}

In the non-archimedean framework, R. Cluckers, G. Comte and F. Loeser then studied the $p$-adic case in \cite{cluckers_local_2012}. The Lebesgue measure is replaced by the Haar measure on $\mathbb{Q}_p$, or more generally a finite extension $K$ of $\Q_p$, and they consider definable sets $X \subset K^n$, either semi-algebraic or subanalytic. Unlike the real and complex case, the sequence of normalized local volumes $\theta_m := \frac{\mu_d(X\cap B_n(x,m))}{\mu_d(B_d(x,m))}$ where $B_d(x,m)$ is the ball of center $x$ and valuative radius $m$, does not always converge. But they show that it has a cyclic convergence, i.e. there exists an integer $e >0$ such that the subsequences $(\theta_{ke+i})_{k \geq 0}$ converge to $d_i$. They then define the local density of $X$ at $x$ as the mean value $\frac{1}{e}\sum_{i=1}^{e-1} d_i$.

\vspace{\baselineskip}

These results have been generalized to the case of Laurent series $\C(\!(t)\!)$, and more generally to discrete valued Henselian fields of characteristic zero by A. Forey \cite{forey_motivic_2017}. Because these fields are not locally compact, there is no Haar measure, but we can use motivic integration, here in its version due to R. Cluckers and F. Loeser \cite{cluckers_constructible_2008} to give a sense to the volume. It is not a classical measure theory as Lebesgue, the volume is no longer a real number but an element of the Grothendieck ring of varieties on the residue field, denoted $\mathcal{M}_\C$ (see Section \ref{section : mot int}). 

\vspace{\baselineskip}

Denote $\mu_d$ the $d$-dimensional measure defined in \cite{cluckers_constructible_2008}, and let $B_n(x,m) = \{y \in \field^n \mid v(x-y) \geq m\}$ the $n$-dimensional ball around $x \in \field^n$ of valuative radius $m \in \Z$. Consider $X$ a variety over $\C$ of dimension $d$ with an isolated singularity $O$, and $\underline{X} \subset h[n,0,0]$ the definable set associated to the $\C[\![t]\!]$-points of $X$. We have: 
$$
\theta_m := \frac{\mu_d(\underline{X}\cap B_n(O,m))}{\mu_d(B_d(O,m))} \in \mathcal{M}_\C
$$
and there exists an integer $e > 0$ such that for any $i = 0, 1, ..., e - 1$, the
subsequence $(\theta_{ke+i})_{k \geq 0}$ converge to some $d_i \in \mathcal{M}_\C$, where the topology is induced by the degree in $\LL$. As in the $p$-adic case, we define the motivic local density as the mean value of the $d_i$. Since the motivic local density is independent of the choice of the embedding (see Proposition \ref{prop : indep embed}), we have that:
$$
\Theta_d^\text{mot}(X,O):= \frac{1}{e} \sum_{i=0}^{e-1} d_i \in \Q \otimes \mathcal{M}_\C.
$$
This periodic convergence can occur even if $X$ is an algebraic set, for example the set $X$ defined by the equation $x^2=y^3$ has a motivic local density at the origin equal to 1/2.

\vspace{\baselineskip}

In \cite{forey:tel-01871909}, A. Forey proved the following proposition: \\

\begin{proposition}
    Let $f \in \C[x,y]$ be a power series without square factor. Let $f=f_1\dots f_r$ the decomposition in $\mathbb{C}[\![x,y]\!]$ of $f$ into irreducible factors. Denote by $N_i$ the multiplicity of $f_i$ and $C$ the curve defined by $f$. Then the motivic density of $C$ at the origin is given by:
    $$
    \Theta^\text{mot}_1(C,0)=\sum_{i=1}^r \frac{1}{N_i}.
    $$
\end{proposition}
We can compare with the complex local density at the origin of the complex curve defined by $f=0$, also known as Lelong's number. It is equal to the multiplicity of $f$, $\sum_{i=1}^r N_i$ by Draper's result \cite{draper_intersection_1969}. 

\vspace{\baselineskip}

In this paper, we prove an explicit formula for the motivic local density of isolated surfaces singularities. We use an additional data, the inner rate related to the bilipschitz geometry of the singularity. Those inner rates were first introduced in \cite{birbrair_thick-thin_2014} by L. Birbrair, W. D. Neumann and A. Pichon, and were studied in \cite{silva_inner_2022} by A. Belotto da Silva, L. Fantini and A. Pichon. They study the metric structure of the germ of an isolated complex surface singularity $(X, 0)$ by means of an infinite family of numerical analytic invariants, called inner rates. A resolution of $(X,0)$ is good if its exceptional locus is a normal crossing divisor. Given a good resolution $\pi:  Y \longrightarrow X$ of $(X,0)$ that factors through the blowup of the maximal ideal and through the Nash modification of $(X,0)$, an irreducible component $E_i$ of the exceptional divisor $\pi^{-1}(0)$ of $\pi$, and a small neighborhood $\mathcal{N}(E_i)$ of $E_i$ in $Y$ with a neighborhood of each double point of $\pi^{-1}(0)$ removed, the inner rate $q_i$ of $E_i$ is a rational number that measures how fast the subset $\pi(\mathcal{N}(E_i))$ of $(X,0)$ shrinks when approaching the origin (see Section \ref{section : taux internes}). The inner rate $q_i$ is independent of the choice of an embedding of $(X,0)$ into a smooth germ, only depending on the analytic type of $(X,0)$ and on the divisor $E_i$. The finite set of inner rates which are bilipschitz invariants for the inner metric has been described explicitly in \cite{birbrair_thick-thin_2014}. It is worth noting that an inner rate equal to 1 characterizes a metrically conical region of the surface, it corresponds to linear inner contact between two complex curve germs on $(X,0)$.

\vspace{\baselineskip}

The main result of this paper is the following theorem:

\vspace{\baselineskip}

\begin{theorem}\label{thm : formula surface}
    Let $X$ be an algebraic surface (i.e. a 2-dimensional separated scheme of finite type) over $\C$ with a unique isolated singularity $O$. Let $\pi : Y \longrightarrow X$ be a good resolution of the singularity that satisfies the previous conditions. Denote $E_i$ the irreducible components of the exceptional locus, $m_i$ its multiplicity and $q_i$ its associated inner rate. Set $E_i^0=E_i \smallsetminus \cup_{j \neq i} E_j$.
    The motivic local density of $X$ at the point $O$ is given by:
    $$
    \Theta^\text{mot}_2(X,O) = \sum_{i\mid q_i=1} \frac{1}{m_i} \Bigg(\frac{1}{\LL+1}[E_i^0] + \sum_{j \mid E_i\cap E_j \neq \emptyset }\frac{\LL^{-(q_{j}-1)m_{j}}(1-\LL^{-1})}{(1-\LL^{-(q_{j}-1)m_{j}})(1+\LL^{-1})} \Bigg),
    $$
    which lies in $\mathcal{M}_\C \otimes \Q$.
\end{theorem}

\vspace{\baselineskip}

The proof is based on a resolution of the singularity. The morphism of resolution allows us to compute the motivic local density using a theorem of change of variable (see Section \ref{section : mot int}, Theorem \ref{thm: change var}). In order to compute the jacobian factor appearing in this change of variable, we prove the following proposition (See Proposition \ref{disc mather}) which links the inner rates and the Mather discrepancies. We use the work of Y. Cherik (See \cite{cherik_inner_2022}) on the inner rates of finite morphisms to prove this proposition. 

\vspace{\baselineskip}

\begin{proposition} 
Let $\pi : Y \longrightarrow X$ be a good resolution of the isolated singularity of $X$, that factors through the Nash transform. Then $q_i$ the inner rate associated to an irreducible component $E_i$ of the exceptional divisor $\pi^{-1}(0)$ is equal to the Mather log-discrepancy normalized by the multiplicity, minus one. In other words, $q_i = \frac{{\hat{k}_i}^\text{log}}{m_i}-1$.
\end{proposition}

\vspace{\baselineskip}

The paper is organized as follow. In the next section, we will recall some basic definitions on motivic integration. Then, we are going to introduce in Section \ref{section : taux internes} the inner rates of a complex surface, based on the work of A. Belotto da Silva, L. Fantini and A. Pichon in \cite{silva_inner_2022}. An important result of this paper is the proposition making the link between those inner rates and the Mather discrepancies, that we will discuss in Section \ref{section : mather}. We finally introduce the motivic local density (see Section \ref{section : properties}) and prove some of its properties. In fact, the motivic local density is a 1-bilipschitz invariant with respect to the non-Archimedean norm. The last section of this paper is dedicated to the proof of Theorem \ref{thm : formula surface}, and we also give a proof of Proposition 1.1, using our method via resolution of singularities, which is different from the method used by A. Forey.

\vspace{\baselineskip}

\textbf{Acknowledgements:} I acknowledge the support of the CDP C2EMPI, together with the French State under the France-2030 programme, the University of Lille, the Initiative of Excellence of the University of Lille, the European Metropolis of Lille for their funding and support of the R-CDP-24-004-C2EMPI project. I gratefully acknowledge R. Cluckers and A. Forey for their guidance, L. Fantini for enlightening discussions and M. Raibaut for his constructive comments on this paper. 

\section{Recap on motivic integration} \label{section : mot int}

Historically, Batyrev showed the equality of Betti Numbers of two complex proper smooth birational Calabi-Yau varieties over $\C$ using p-adic integration and Weil conjectures for zeta functions. Kontsevich constructed the motivic integral in order to generalize this result and showed the equality of Hodge Numbers. Several motivic integration theories had been developped in the last thirty years. The sets we want to measure are the semi-algebraic sets in valued fields, they are seen as subsets of the arc space of variety over the residue field in Kontsevich/Denef-Loeser theory, and as definable sets in a first order theory in the Cluckers-Loeser or Hrushovski-Kazhdan version.

\vspace{\baselineskip}

We are considering the theory of Cluckers-Loeser (See \cite{cluckers_constructible_2008}), specialized to the case of algebraically closed residue field. Let us recall the setting. Fix $k$ a field of characteristic zero. Consider the three sorted language of Denef-Pas, with one sort for the valued field, with the language of rings $\{0,1, +,-, \cdot\}$, one sort for the residue field with the language of rings and one sort for the value group with the Presburger language $\{0,+,\leq,\{\equiv_n\}_{n \geq 1} \}$ where $\equiv_n$ is a binary relation for congruence modulo $n$. We also add symbol $v$ for the valuation map, $\ac$ for the angular component, and constants symbols for elements $k(\!(t)\!)$. We will use this language to describe the Henselian discretely valued field of characteristic zero. We will consider structures for $\mathcal{L}_{\text{DP}}$ consisting of tuples $(K,k,\Gamma)$ where $K$ is a
valued field with value group $\Gamma$, and residue field $k$. For example, for all field $k$, $k(\!(t)\!)$  the field of formal Laurent series over $k$ is an $\mathcal{L}-$structure $(k(\!(t)\!),k,\Z)$, where the symbols are naturally interpreted: 
$$
v\left(\sum_{i \geq i_0}^\infty  a_it^i \right) = i_0, ~ \ac \left( \sum_{i \geq i_0}^\infty  a_it^i \right) = a_{i_0} \text{ if } a_{i_0} \neq 0.
$$

The $\mathcal{L}_\text{DP}$-theory of the above discribed structures whose valued field is Henselian and whose residue field is of characteristic zero admits elimination of quantifiers in the valued field sort. 

\vspace{\baselineskip}

\begin{definition}(See Section 2.3 of \cite{cluckers_constructible_2008})
    We define $h[m,n,r]$ the functor from the category of fields containing $k$ to the category of sets sending an algebraically closed field $F$ to 
    $$
    h[m,n,r](F)=F(\!(t)\!)^m \times F^n \times \Z^r.
    $$
    For $X$ an affine variety over $F$ in $\mathbb{A}^n_F$, denote $\underline{X}$ the subdefinable set of $h[n,0,0]$, which to any algebraically closed field $F$ associates $X(F[\![t]\!])$.
\end{definition}

\vspace{\baselineskip}

In this paper, definable means definable in the Denef-Pas language. \\

\begin{definition}
    If $D$ is a subassignment of $h[m,n,r]$, we define $\text{dim}(D)$ as the dimension of the image of $D$ under the projection $h[m,n,r] \longrightarrow h[m,0,0]$.
\end{definition}

\vspace{\baselineskip}

\begin{definition}
    Let $k$ be a field, the Grothendieck group of varieties $K_0(\text{Var}_k)$ is the free abelian group generated by symbols $[X]$ for the class of isomorphism of $X$ a variety (a separated scheme of finite type over $k$), with relations  $[X] = [Y]+[X \smallsetminus Y]$, where $Y$ is a closed subvariety. The fiber product induces a ring structure on $K_0(\text{Var}_k)$ with identity element the class of a point. Set $\LL:=[\mathbb{A}_k^1]$ the class of the affine line.
\end{definition} 

\vspace{\baselineskip}

\begin{example}
    It follows from the definition that $[\mathbb{P}^1_k]=\LL+1$, and by inductively using the usual affine cover of $\PP^n_k$, then $[\PP^n_k]= 1+\LL+\dots+\LL^n$.
\end{example}

\vspace{\baselineskip}

We consider the localization
$$
\mathcal{M}_k := K_0(\text{Var}_k)\left[ \LL^{-1}, \left\{ \frac{1}{1-\LL^{-i}} \right\}_{i >0} \right].
$$
The topology on $\mathcal{M}_k$ is induced by the $\LL-$degree. \\

In \cite{cluckers_constructible_2008}, Cluckers and Loeser assign to some definable subset $D$ of dimension $d$ an element of $\mathcal{M}_k$, called the motivic volume of $D$, that we denote by $\mu_d(D)$. They define more generally a whole class of socalled motivic constructible functions that are stable under integration and with good properties like Fubini’s theorem and change of variable formula. For example, we want to integrate over a definable set $A$ functions $\LL^\alpha$, where $\alpha$ is some definable $\Z$-valued function on $A$. In this paper, we will use the change of variable formula from \cite{cluckers_constructible_2008}, in the particular case of birational morphism between algebraic varieties (see also Lemma 3.3 of \cite{denef_germs_1999}): 

\vspace{\baselineskip}

\begin{theorem}(Theorem 12.1.1 of \cite{cluckers_constructible_2008}, or Lemma 3.3 of \cite{denef_germs_1999})\label{thm: change var}
    Let $h : Y \longrightarrow X$ a proper birational morphism between two algebraic varieties over $k$ of pure dimension $d$, consider $A$ a definable set of $\underline{X}$, $\alpha : A \longrightarrow \N$ a simple function. Then we have the following equality:
    $$
    \int_{A} \LL^{-\alpha}|\omega_X| = \int_{h^{-1}(A)} \LL^{-\alpha \circ h - v(h^*(\Omega_X^d))}|\omega_Y|,
    $$
    where $|\omega_X|,|\omega_Y|$ are the canonical volume forms associated to $\underline{X}$ and $\underline{Y}$. (See Section \ref{section : properties} for the definition of the canonical volume form.)
\end{theorem}

\vspace{\baselineskip}

\begin{remark}
    When $X$ is a smooth $k-$variety of pure dimension $d$, then $\mu_d(\underline{X})=[X]\LL^{-d}$. 
\end{remark}

\section{Introduction to the inner rates of a complex surface}\label{section : taux internes}

In this section, we recall some definitions and results of \cite{silva_inner_2022} about the metric structure of the germ of an isolated complex
surface singularity $(X, 0)$. \\

We recall the big-Theta asymptotic notations of Bachmann-Landau in the following form: given two function germs $f,g : ([0,\infty),0) \longrightarrow [0,\infty)$ we say $f$ is \textit{big-Theta} of $g$ and we write $f(t)=\varTheta(g(t))$ if there exist real numbers $\eta >0$ and $K>0$ such that for all $t$, if $f(t) \leq \eta$ then $K^{-1}g(t)\leq f(t) \leq Kg(t)$. \\

Let $(\gamma,0)$ and $(\gamma',0)$ be two distinct germs of complex curves on the surface germ $(X,0)$. Denote by $d_{\text{inn}}$ the inner distance on $(X,0)$, obtained by taking infimum of length of paths on $X$ between the two points. Denote by $\mathbb{S}_\varepsilon$ the sphere in $\C^n$ centered in 0 and of radius $\varepsilon$. The inner contact between $\gamma$ and $\gamma'$ is the rational number $q_i$ such that:
$$
d_{\text{inn}}(\gamma \cap \mathbb{S}_\varepsilon,\gamma' \cap \mathbb{S}_\varepsilon) = \varTheta(\varepsilon^{q_i}).
$$
The existence and the rationality of $q_i$ is recalled in Lemma \ref{lem : inner rate}. \\

Recall that if $\pi : Y \longrightarrow X$ is a resolution of $X$ and if $E$ is an irreducible component of $\pi^{-1}(0)$, a curvette of $E$ is a smooth curve germ $(\gamma,p)$ in $Y$, where $p$ is a point of $E$ which is a smooth point of $\pi^{-1}(0)$ and such that $\gamma$ and $E$ intersect transversaly.
We call \textit{good} resolution of $(X,0)$ a proper morphism $\pi : Y \longrightarrow X$ such that $Y$ is regular, $\pi$ is an isomorphism outside of its excepional locus $\pi^{-1}(0)$, and the latter is a strict normal crossing divisor on $Y$. The fact that the exceptional divisor has normal crossing allows us to associate with it its dual graph $\Gamma_\pi$, which is the graph whose vertices are in bijection with the irreducible components of $\pi^{-1}(0)$, and where two vertices of $\Gamma_\pi$ are joined by an edge for each point of intersection of the corresponding components.

\vspace{\baselineskip}

\begin{definition}
    Denote by $\mathcal{O}=\widehat{\mathcal{O}_{X,0}}$ the completion of the local ring of $X$ at 0 with respect to its maximal ideal. A semivaluation on $\mathcal{O}$ is a map $v : \mathcal{O} \longrightarrow \R_+ \cup \{+\infty\}$ such that, for every $f,g$ in $\mathcal{O}$ and every $\lambda \in \C$, we have:
    \begin{enumerate}
        \item $v(fg)=v(f)+v(g)$,
        \item $v(f+g) \geq \text{min}\{v(f),v(g)\}$,
        \item $v(\lambda)= \left\{
    \begin{array}{ll}
        +\infty & \mbox{if } \lambda =0 \\
        0 & \mbox{otherwise.}
    \end{array}
\right.$
    \end{enumerate}
\end{definition}

\vspace{\baselineskip}

\begin{example}(Example 2.1 of \cite{silva_inner_2022})
    Let $\pi :Y \longrightarrow X$ be a good resolution of $(X,0)$ and let $E$ be an irreducible component of the exceptional divisor $\pi^{-1}(0)$. Then the map:
    $$
    \begin{tabular}{ccccc}
        $v_E :$ & $\mathcal{O}$ & $\longrightarrow$ & $\R_+ \cup \{+\infty\}$\\
         & $f$ & $\mapsto$ & $\frac{\text{ord}_E(f)}{\text{ord}_E(\mathfrak{M})}$
    \end{tabular}
    $$
    is a valuation on $\mathcal{O}$, where $\mathfrak{M}$ is the maximal ideal of $\mathcal{O}$, called the divisorial valuation associated with $E$. This valuation is independent of the choice of $\pi$. If $v$ is the divisorial valuation associated with a component of the exceptional divisor $\pi^{-1}(0)$ of a good resolution $\pi : Y \longrightarrow X$ of $(X,0)$, we denote by $E_v$ this prime divisor, and by $m_v$ the positive integer $\text{ord}_{E_v}(\mathfrak{M})$, called the multiplicity of $E_v$. This terminology is justified by the fact that $m_v$ is also the multiplicity of $E_v$ in $\pi^{-1}(0)$ considered with its natural scheme structure.
\end{example}

\vspace{\baselineskip}

\begin{definition}
    The \textit{non-archimedean link} $\text{NL}(X,0)$ of $(X,0)$ is the topological space whose underlying set is 
    $$
    \{v : \mathcal{O} \longrightarrow \R_+\cup\{+\infty\} \text{ semi-valuation } \mid v(\mathfrak{M})=\{1\} \}
    $$
    where $\mathcal{O}=\widehat{\mathcal{O}_{X,0}}$ the completion of the local ring of $X$ at 0 with respect to its maximal ideal $\mathfrak{M}$.
    Its topology is induced from the product topology of $(\R_+\cup \{+\infty\})^\mathfrak{M}$.
\end{definition}

\vspace{\baselineskip}

\begin{remark}
    If $(X,0)=(\C^2,0)$ is a smooth surface germ, its non-archimedean link $\text{NL}(\C^2,0)$ is the \textit{valuative tree}, that was studied in \cite{favre_valuative_2002}.
\end{remark}

\vspace{\baselineskip}

Given any good resolution $\pi : Y \longrightarrow X$ of $(X,0)$ with dual graph $\Gamma_\pi$, there exist a natural embedding 
$$
i_\pi : \Gamma_\pi \hookrightarrow \text{NL}(X,0)
$$
and a canonical continuous retraction 
$$
r_\pi : \text{NL}(X,0) \longrightarrow \Gamma_\pi
$$
such that $r_\pi \circ i_\pi = \text{id}_{\Gamma_\pi}$. 
The continuous retractions $r_\pi$ induce a natural homeomorphism 
$$
\text{NL}(X,0) \overset{\simeq}{\longrightarrow} \lim\limits_{\overset{\longleftarrow}{\pi}} \Gamma_\pi.
$$
where $\pi$ ranges over the filtered family of good resolutions of $(X,0)$. This means that the non-archimedean link $\text{NL}(X,0)$ can be thought of as a universal dual graph, making it a very convenient object for studying the
combinatorics of the resolutions of $(X,0)$. One can find more details about the structure of the non-archimedean link in terms of the dual graph of the good resolutions in Section 2.1 of \cite{silva_inner_2022}.

\vspace{\baselineskip}

\begin{lemma}[3.2 of \cite{silva_inner_2022}]\label{lem : inner rate}
Let $v \in \text{NL}(X,0)$ be a divisorial valuation on $(X,0)$ and let $\pi : Y \longrightarrow X$ be a good resolution of $(X,0)$ which factors through the blowup of the maximal ideal and through the Nash transform of $(X,0)$ and such that $v$ is the divisorial valuation associated with an irreducible component $E_v$ of $\pi^{-1}(0)$. Consider two curvettes $\gamma^*$ and $\tilde{\gamma}^*$ of $E_v$ meeting it a distinct points, and write $\gamma = \pi(\gamma^*)$ and $\tilde{\gamma} = \pi(\tilde{\gamma}^*)$. Then the inner contact $q_i^X(\gamma,\tilde{\gamma})$ between $\gamma$ and $\tilde{\gamma}$ only depends on $v$ and not on the choice of $\pi, \gamma^*$ and $\tilde{\gamma}^*$. Moreover, $m_vq_i^X(\gamma,\tilde{\gamma})$ is an integer, $q_i^X(\gamma,\tilde{\gamma}) \geq 1$, and if $l : (X,0) \longrightarrow (\C^2,0)$ is a generic projection with respect to $\pi$ we have $q_i^X(\gamma,\tilde{\gamma})=q_i^{\C^2}(l(\gamma),l(\tilde{\gamma}))$.
\end{lemma}

\vspace{\baselineskip}

\begin{definition}
We denote by $q_v$ the rational number $q_i^X(\gamma,\tilde{\gamma})$ and call it the \textit{inner rate} of $v$. In the setting of a good resolution, every irreducible component $E_i$ of the exceptional divisor $\pi^{-1}(0)$ determines a valuation $v_i$; to simplify notations throughout the paper, we write $q_i:=q_{v_i}$.
\end{definition}

\vspace{\baselineskip}

We have the following computation lemma for the inner rate of divisorial valuation of a singular germ $(X,0)$.

\vspace{\baselineskip}

\begin{lemma}[3.6 of \cite{silva_inner_2022}] \label{calcul taux}
Let $\pi : Y \longrightarrow X$ be a good resolution of $(X,0)$ which factors through the blowup of the maximal ideal and through the Nash transform of $(X,0)$. Let $p$  be a point of the exceptional divisor $\pi^{-1}(0)$ and let $E_w$ be the exceptional component created by the blowup of $Y$ at $p$. Then: 
\begin{enumerate}[label=$(\roman*)$]
\item If $p$ is a smooth point of $\pi^{-1}(0)$ and $E_v$ is the irreducible component of $\pi^{-1}(0)$ in which $p$ lies, then:
$$
m_w = m_v \text{ and } q_w = q_v + \frac{1}{m_v}.
$$
\item If $p$ lies in the intersection of two irreducible components $E_v$ and $E_{v'}$ of $\pi^{-1}(0)$, then: 
$$
m_w = m_v+m_{v'} \text{ and } q_w = \frac{m_vq_v+m_{v'}q_{v'}}{m_v+m_{v'}}.
$$
\end{enumerate}
 
\end{lemma}

\begin{lemma}[3.8 of \cite{silva_inner_2022}]
    There exist a unique continuous function 
    $$
    \mathcal{I}_X : \text{NL}(X,0) \longrightarrow \R_{\geq 1} \cup \{\infty\}
    $$
    such that $\mathcal{I}_X(v)=q_v$ for every divisorial point $v$ of $\text{NL}(X,0)$. If $\pi$ is a good resolution of $(X,0)$ that factors through the blowup of the maximal ideal and the Nash transform of $(X,0)$, then $\mathcal{I}_X$ is linear on the edges of $\Gamma_\pi$ with integral slopes.
\end{lemma}

\section{Discrepancies and the Mather canonical divisor}\label{section : mather}

In this section, we set $k=\C$ in order to make the link with the complex Lipschitz geometry. We will link two important notions for this paper: the Mather discrepancies and the inner rates. In the smooth case, we prove in Lemma \ref{lem: disc smooth} that the inner rate function $\mathcal{I}_{\C^2}$ is equal to the normalized log discrepancy function on $\text{NL}(\C^2,0)$, minus one. In the singular case, the inner rate can be related the Mather log discrepancy, as conjectured in the Remark 3.11 of \cite{silva_inner_2022}. We prove this in Proposition \ref{disc mather}. We will use those results in order to compute the order of the jacobian that appears in the change of variables formula while computing the local density via a morphism of resolution.

\vspace{\baselineskip}

\begin{lemma} \label{lem: disc smooth}
Consider the germ $(\C^2,0)$, then the inner rates $q_i$ associated to an irreducible component $E_i$ of the exceptional divisor is equal to the log-discrepancy normalized by the multiplicity $m_i$, minus one. In other words, $q_i = \frac{k_i^\text{log}}{m_i}-1$, where $k_i^\text{log} = k_i +1$ and $k_i = \text{ord}_{E_i}(K_{Y/\C^2})$ for $Y$ the surface obtained after blowing up, and where $K_{Y/\C^2}$ is the relative canonical divisor. 
\end{lemma}

\vspace{\baselineskip}

\begin{proof}
We prove the result by induction on the number of blow-ups. Consider 
$$
X_{n-1} \xrightarrow{\pi_{n-1}} \dots \longrightarrow X_1 \xrightarrow{\pi_1} \C^2.
$$
Up to reordinaring, one may assume that: either $\pi_n$ is the blow-up of $p \in E_{n-1}^0:=E_{n-1} \smallsetminus \cup_{j \neq n-1}E_j$, or $\pi_n$ is the blow-up of $E_{n-1}\cap E_{n-2}$ with convention that $E_0^0=\{0\}$. 

If $\pi_n$ is the blowup created by blowing up a point $p \in E_{n-1}$ the exceptional divisor created by $\pi_{n-1}$, we have the following equality: 
$$
K_{X_n/X_{n-1}} = E_n = K_{X_n} - \pi_n^* K_{X_{n-1}}
$$
where $K_{X_{n-1}} = \sum_{i=1}^{n-1} k_{i} E_i$ is the canonical divisor of $X_{n-1}$ and $E_n$ is the exceptional divisor created by $\pi_n$. Hence, 

\begin{align*}
K_{X_n} & = E_n + \pi_n^* K_{X_{n-1}} \\
& = (1+k_{n-1})E_n + \sum_{i=1}^{n-1} k_{i} E_i \\
& = K_{X_n/\C^2}
\end{align*}

With Lemma \ref{calcul taux}, we have that $q_{n} = q_{n-1} +1$, and by the hypothesis of induction, $q_{n-1} = \frac{k_{n-1}+1}{1}-1$. Hence, $q_{n} = k_{n-1}+1 = k_{n}$ which is equal to $\frac{k_{n}+1}{1}-1$ the property at rank $n$.

\vspace{\baselineskip}

Now consider $\pi_n$ the blowup created by the blowup of $E_{n-1} \cap E_{n-2}$. Then:
\begin{align*}
K_{X_n} & = E_n + \pi_n^* K_{X_{n-1}} \\
& = (1+k_{n-1}+k_{n-2})E_n + \sum_{i=1}^{n-1} k_{i} E_i \\
\end{align*}

Using Lemma \ref{calcul taux}, we have that $q_{n} = \frac{q_{n-1}+q_{n-2}}{2}$ and by the induction hypothesis, $q_{n} = \frac{k_{n-1}+k_{n-2}}{2} = \frac{k_{n}-1}{2}$. The property we want at rank $n$ is $q_{n} = \frac{k_{n}+1}{2}-1$, which we have.

\end{proof}

In the singular case, we need to modify the definition of discrepancy. Given an arbitrary variety $X$, we take a resolution of singularities $\pi : Y\longrightarrow X$ that factors through the Nash blow-up and define the relative Mather canonical divisor $\hat{K}_{Y/X}$.  This divisor, which is defined in total generality, is always an effective integral divisor, and it coincides with $K_{Y/X}$ when $X$ is smooth (the two divisors are in general different for
$\Q$-Gorenstein varieties). We recall the following definitions, assuming that all varieties are irreducible and reduced,
over a fixed algebraically closed field $k$ of characteristic zero, see \cite{ishii_mather_2011}, \cite{ein_multiplier_2011} for more details. \\

Let $X$ be a $\Q-$Gorenstein variety of index $r$ (i.e. normal and its canonical class $K_X$ is $\Q-$Cartier, so $rK_X$ is Cartier), and $\pi : Y \longrightarrow X$ a resolution of singularities of $X$. Then the usual discrepancy divisor $K_{Y/X}$ is the unique $\Q$-Cartier divisor supported on the exceptional locus of $\pi$ such that $rK_{Y/X}$ is linearly equivalent with $rK_Y-\pi^*(rK_X)$. Note that the usual discrepancy is defined only for a $\Q-$Gorenstein variety $X$, and the following Mather discrepancy is defined for every variety, even for non-normal variety. 

\vspace{\baselineskip}

\begin{definition}
    Let $X$ be an $n$-dimensional variety. The sheaf $\Omega_X^n = \wedge^n\Omega_X$ is invertible over the smooth locus $X_\text{reg}$ of $X$, hence the morphism 
    $$
    \pi : \PP(\Omega_X^n) = \text{Proj}(\text{Sym}(\Omega_X^n)) \longrightarrow X
    $$
    is an isomorphism over $X_\text{reg}$. The \textit{Nash blow-up} $\nash$ is the closure of $\pi^{-1}({X_\text{reg}})$ in $\PP(\Omega_X^n)$ (with the reduced scheme structure). Note that by construction we have a projective birational morphism $\nu : \nash \longrightarrow X$ and a surjective morphism $\nu^*(\Omega_X^n) \longrightarrow \mathcal{O}_{\PP(\Omega_X^n)}(1)_{|\nash}$. Set $\sheaf := \mathcal{O}_{\PP(\Omega_X^n)}(1)_{|\nash}$, it is an ample line bundle over $\hat{X}$. Furthermore, this is universal in the following sense: given a birational morphism of varieties $\phi :Y \longrightarrow X$, this factors through $\nu$ if and only if $\phi^*(\Omega_X^n)$ admits a morphism onto an invertible sheaf $\mathcal{L}$ on $Y$ (which identifies $\mathcal{L}$ to the quotient of $\phi^*(\Omega_X^n)$ by its torsion). 
\end{definition}

\vspace{\baselineskip}

\begin{definition}\label{def:mather disc}
Let $X$ be a variety of dimension $n$ and $f : Y \longrightarrow X$ be a resolution of the singularities, factoring through the Nash blow-up. Then, the image of the canonical homomorphism 
$$
f^* \wedge^n\Omega_X \longrightarrow \wedge^n\Omega_Y
$$
is an invertible sheaf of the form $J\wedge^n\Omega_Y$, where $J$ is the invertible ideal sheaf of $\mathcal{O}_Y$ that defines an effective divisor supported on the exceptional locus of $f$. This divisor is called the \textit{Mather discrepancy divisor} and denoted by $\hat{K}_{Y/X}$. For every prime divisor $E$ on $Y$, we define 
$$
\hat{k}_E:= \text{ord}_E(\hat{K}_{Y/X})
$$
and
$$
\hat{k}_E^\text{log}:= \hat{k}_E +1.
$$
\end{definition}

\vspace{\baselineskip}

\begin{remark}
If $X$ is smooth, then $\hat{k}_E = k_E$ and $\hat{K}_{Y/X} = K_{Y/X}$. 
\end{remark}

\vspace{\baselineskip}

\begin{definition}
    The \textit{Jacobian ideal} of a variety $X$ denoted $Jac_X$ is the 0-th Fitting ideal $\text{Fitt}_0(\Omega_X)$ of $\Omega_X$. This is an ideal whose support is the
singular locus of $X$. If $f : Y \longrightarrow X$ is a log resolution of $Jac_X$, we denote by $J_{Y/X}$ the effective divisor on $Y$ such that $Jac_X \cdot \mathcal{O}_Y = \mathcal{O}_Y(-J_{Y/X} )$.
\end{definition}

\vspace{\baselineskip}

\begin{remark}[Remark 2.7 of \cite{ein_multiplier_2011}]
    If $X$ is normal and locally a complete intersection, then $\hat{K}_{Y/X}-J_{Y/X}=K_{Y/X}$ the usual discrepancy divisor.
\end{remark}

\vspace{\baselineskip}

The goal of this section is to prove the Proposition \ref{disc mather}, that will show the link between the Mather discrepancies and the inner rates. One can also see the proof of Theorem A of Y. Cherik in \cite{cherik_inner_2022}.

\vspace{\baselineskip}

\begin{lemma}\label{lem: ord=mather}
    Let $X$ be a surface with an isolated singularity, and let $\pi : Y \longrightarrow X$ be a resolution of singularity that factors through the Nash blow-up. Let $l: X \longrightarrow \C^2$ be a generic projection, set $\omega :=l^*dx\wedge dy \in \Omega_X$. Then the divisor associated to the $2-$form $\pi^* \omega$ is represented by:
    $$
    K_{\pi^* \omega}=\sum_{i=1}^n \hat{k}_iE_i + \Pi_l^*,
    $$
    where $E_i, i =1,\dots,n$ are the irreducible components of the exceptional divisor $\pi^{-1}(0)$, and $\Pi_l^*$ is the strict transform of the polar curve defined by $\Pi_l=\overline{\text{Crit}(l_{|X_\text{reg}})}$.

    In particular, $\text{ord}_{E_i}(\pi^*\omega)=\hat{k}_{E_i}=:\hat{k}_{i}$ for $E_i$ an exceptional component of the exceptional divisor.
\end{lemma}

\vspace{\baselineskip}

\begin{proof}
    We have the following diagram:
    $$
    \xymatrix{ Y \ar@/^2pc/[rr]^\pi \ar[r]^\mu & \hat{X} \ar[r]^\nu & X  }
    $$
    Since $dx \wedge dy$ is a  non-vanishing 2-form on $\C^2$, then $\omega$ is a rational 2-form, regular on $X_\text{reg}$. By definition of the Nash blow-up, $\nu^*\omega$ is a global section in $\Gamma(\nash,\sheaf)$ (above the singular point, it is defined as a limit of tangent planes). Pulling back via the map $\mu$, we get that $\pi^*\omega \in \Gamma(Y,\mu^*\sheaf)$. 

    We now prove the following isomorphism: $\mu^*\sheaf \simeq \tensor$. Since $\nu^*\sheaf \longrightarrow \sheaf$ is surjective, then $\pi^*\Omega_X^2 \longrightarrow \mu^*\sheaf$ is also surjective. By the definition \ref{def:mather disc}, $\pi^*\Omega_X^2 \twoheadrightarrow \text{Im}(\pi^*\Omega_X^2 \longrightarrow \Omega_Y^2) = \tensor$. By universality of the Nash blow-up, we can conclude that $\mu^*\sheaf \simeq \tensor$. 

    Finally, the 2-form $\pi^*\omega$ is a global section of the invertible sheaf $\tensor$, hence $\text{div}(\pi^*\omega)\sim \hat{K}_{Y/X}$. The strict transform of the polar curve of $l$ gives the non-exceptional part of the divisor of $\pi^* \omega$, and the exceptional contribution is $\hat{K}_{Y/X}$. Indeed, the divisor $K_{\pi^* \omega}$ is represented by
    $$
    K_{\pi^* \omega}=\sum_{i=1}^n\hat{k}_iE_i + \Pi_l^*
    $$
    where $E_i, i =1,\dots,n$ are the irreducible components of the exceptional divisor $\pi^{-1}(0)$, $\Pi_l^*$ is the strict transform of the polar curve defined by $\Pi_l=\overline{\text{Crit}(l_{|X_\text{reg}})}$.
\end{proof}

\vspace{\baselineskip}

\begin{proposition} \label{disc mather}
Let $\pi : Y \longrightarrow X$ be a good resolution of the isolated singularity of $X$, that factors through the Nash transform. Then $q_i$ the inner rate associated to an irreducible component $E_i$ of the exceptional divisor $\pi^{-1}(0)$ is equal to the Mather log-discrepancy normalized by the multiplicity, minus one. In other words, $q_i = \frac{{\hat{k}_i}^\text{log}}{m_i}-1$.
\end{proposition}

\vspace{\baselineskip}

\begin{proof}
    We recall elements of the proof of the Proposition 3.2 of \cite{cherik_inner_2022}. Let $E_i$ be an irreducible component of $D=\pi^{-1}(0)$, and $l=(l_1,l_2) : X \longrightarrow \C^2$ the generic projection.
    Set $\omega :=l^*dx\wedge dy$. We consider $p$ a smooth point of $D$ in $E_i$, with $p \notin l_1^*\cup l_2^* \cup \Pi_l^*$, where $l_1^*,l_2^*$ and $\Pi_l^*$ are the strict transform of $l_1,l_2$ and the polar curve $\Pi_l$. The divisor $K_{\pi^*\omega}$ associated to the meromorphic 2-form is represented by:
    $$
    K_{\pi^*\omega} = \sum_{i \in V(\Gamma_\pi)}(q_i(m_i+1)-1)E_i + \Pi_l^*.
    $$
    In particular, $\text{ord}_{E_i}(\pi^*\omega)=m_i(q_i+1)-1$. By the previous proposition, we have that $\hat{k}_i=m_i(q_i+1)-1$. Hence, $\hat{k}_i^\text{log}=m_i(q_i+1)$.
\end{proof}

\vspace{\baselineskip}

\begin{remark}
Lemmas \ref{lem: disc smooth} and \ref{disc mather} are proving the remark 3.11 of \cite{silva_inner_2022}.
\end{remark}

\section{Properties of the motivic local density}\label{section : properties}

Before proving the main theorem of this paper, we recall the definition of the motivic local density, and prove some of its properties. The local density have been studied in the case of the Laurent series field $\field$, and more generally in the case of discrete valued Henselian fields of characteristic zero by A. Forey in \cite{forey_motivic_2017}. 

\vspace{\baselineskip}

\begin{definition}
    Say that a function $h : \mathbb{N} \longrightarrow \mathcal{M}_k$ has a mean value at infinity if there exists an integer $e > 0$ such that 
    $$
    \lim\limits_{\overset{n \longrightarrow \infty}{n \equiv c \text{ mod } e}} h(n)
    $$
    exists in $\R$ for each $c=0,\dots,e-1$ and in this case define the mean value at infinity of $h$ as the average 
    $$
    \text{MV}_\infty(h):=\frac{1}{e} \sum_{i=0}^{e-1} \lim\limits_{\overset{n \longrightarrow \infty}{n \equiv c \text{ mod } e}} h(n).
    $$
    The mean value is independent of the choice of the modulus $e>0$. We recall that the topology on $\mathcal{M}_k$ is induced by the $\LL-$degree.
\end{definition}

\vspace{\baselineskip}

Consider $D \subset h[m,0,0]$ a definable set of dimension $d$. Consider $x$ a point of $D$ and set $\theta_d(D,x)(n):= \frac{\mu_d(D \cap B_m(x,n))}{\mu_d(B_d(x,n))} \in \mathcal{M}_k$. 

\vspace{\baselineskip}

\begin{definition}
    The function $\theta_d(D,x)$ admits a mean value at infinity, and we define the local density of the set $D$ at the point $x$ by: 
    $$
    \Theta_d(D,x):= \text{MV}_\infty(\theta_d(D,x)) \in \Q \otimes \mathcal{M}_k.
    $$
    (See Proposition-Definition 3.6 of \cite{forey_motivic_2017} for more details.)
\end{definition}

\vspace{\baselineskip}

With the properties of the motivic measure, we trivially have the following proposition: 

\vspace{\baselineskip}

\begin{proposition}
    Consider $D \subset h[m,0,0]$ a definable set of dimension $d$, $U \subset D$ a subset of $D$ of dimension $<d$ and $x$ a point of $D$. Hence
    $$
    \Theta_d(D,x) = \Theta_d(D \smallsetminus U,x).
    $$
\end{proposition}

\vspace{\baselineskip}

The following proposition shows that renormalizing with balls instead of with spheres yields the same local density. This will helps to make some computations easier.

\vspace{\baselineskip}

\begin{proposition}
    Let $D \subset h[m,0,0]$ be a definable set of dimension $d$ and $x$ be a point of $D$, We have the following equality:
    \begin{align*}
    \Theta_d(D,x) & = \text{MV}_\infty\Big(n\mapsto \frac{\mu_d(D\cap B_m(x,n))}{\mu_d(B_d(x,n))}\Big) \\
    &= \text{MV}_\infty\Big(n\mapsto \frac{\mu_d(D\cap S_m(x,n))}{\mu_d(S_d(x,n))}\Big).
    \end{align*}
\end{proposition}

\vspace{\baselineskip}

We recall the following proposition, providing a formula for the motivic local density of a curve. It was proved by A. Forey in \cite{forey_motivic_2017}, and we will give an other proof of this proposition using a resolution of singularity, similarly as what we will do for the case of surfaces.

\vspace{\baselineskip}

\begin{proposition}[\cite{forey:tel-01871909}] \label{prop : courbes}
    Let $f \in \C[x,y]$ be a power series without square factor. Let $f=f_1\dots f_r$ the decomposition in $\mathbb{C}[\![x,y]\!]$ of $f$ into irreducible factors. Denote by $N_i$ the multiplicity of $f_i$ and $C$ the curve defined by $f$. Then the motivic density of $C$ at the origin is given by:
    $$
    \Theta^\text{mot}_1(C,0)=\sum_{i=1}^r \frac{1}{N_i}.
    $$
\end{proposition}

\vspace{\baselineskip}

We will prove this proposition in the next section (see Proposition \ref{prop : curve section 6}), since we use a lot of tools defined then. 

\vspace{\baselineskip}

With the definitions above, we can define the local density for a definable subset of $\field^m$. In the context of our main theorem, we consider $X$ a surface over $\C$. Up to the choice of an embedding in the affine space, we can consider $\underline{X}$ the definable set associated to the $\C[\![t]\!]-$points of $X$ and compute its motivic local density. The Proposition \ref{prop : indep embed} shows that the local density is independent of the choice of the embedding.

\vspace{\baselineskip}

Let $h$ be a definable subassignment of $h[m,n,r]$ of dimension $d$. We denote by $x_1,\dots,x_m$ the coordinates on $\mathbb{A}^m_{k}$ and we consider the $d$-forms $\omega_I := dx_{i_1} \wedge \dots \wedge dx_{i_d}$ for $I=\{i_1,\dots,i_d\} \subset \{1,\dots,m\}$, $i_1<\dots<i_d$. We denote by $|\omega_I|_h$ the image of $\omega_I$ in $|\tilde{\Omega}|_+(h)$. We recall the following definitions from \cite{cluckers_constructible_2008}.

\vspace{\baselineskip}

\begin{definition}[8.3.1 of \cite{cluckers_constructible_2008}]
    There is a unique element $|\omega_0|_h$ in $|\tilde{\Omega}|_+(h)$, the canonical volume form, such that, for every $I$, there exists $\Z$-valued definable functions $\alpha_I$ and $\beta_I$ on $h$, with $\beta_I$ only taking as values 1 and 0, such that $\alpha_I + \beta_I >0$ on $h$, $|\omega_I|_h=\beta_I\LL^{-\alpha_I}|\omega_0|_h$ in $|\tilde{\Omega}|_+(h)$, and such that $\text{inf}_I\alpha_I=0$. We call $|\omega_0|_h$ the canonical volume form on $h$.\label{def : volume form embed}
\end{definition}

%\vspace{\baselineskip}

%\begin{definition}\label{def : volume form modele}
    %Let $\mathcal{X}^0$ be an algebraic variety over $\text{Spec}(k[\![t]\!])$, flat over $\text{Spec}(k)$. Set $\mathcal{X}:= \mathcal{X}^0 \otimes_{\text{Spec}(k)} \text{Spec}(k(\!(t)\!))$, assume $\mathcal{X}$ of dimension $d$. We denote by $U^0$ the largest open subset of $\mathcal{X}^0$ on which the sheaf ${\Omega_{\mathcal{X}^0}^d}_{|k[\![t]\!]}$ is locally free of rank 1 over $k[\![t]\!]$. Its generic fiber $U := U^0 \otimes_{\text{Spec}(k[\![t]\!])}k(\!(t)\!)$ may be identified with the smooth locus of $\mathcal{X}$ when $\mathcal{X}$ is of pure dimension $d$. We choose a finite cover of $U^0$ by open subsets $U^0_i$ on which the sheaf ${\Omega_{\mathcal{X}^0}^d}_{|k[\![t]\!]}$ is generated by a non zero form $\omega_i$ in ${\Omega_{U^0}^d}_{|k[\![t]\!]}(U_i^0)$. Each form $\omega_i$ gives rise to a volume form $|\omega_i|$ in $|\Omega|_+(h_{U_i})$, where $U_i$ denotes the generic fiber of $U_i$. The subsets $U_i$ form an open cover of $U$. There exist a unique element $|\omega_0|$ in $|\tilde{\Omega}|_+(h_\mathcal{X})$ such that $|\omega_0|_{|h_{U_i}}=|\omega_i|$ in $|\tilde{\Omega}|_+(h_{U_i})$; and only depends on the model $\mathcal{X}^0$, not on the choice of the cover by open subsets $U^0_i$. 
%\end{definition}

\vspace{\baselineskip}

\begin{proposition}\label{prop : indep embed}
    Consider $X$ a variety over $k$ of dimension $d$ with an isolated singularity $O$. The motivic local density $\Theta_d(\underline{X},O)$ is independent of the choice of an embedding in the affine space, and we write it $\Theta_d^\text{mot}(X,O)$. 
\end{proposition}

\vspace{\baselineskip}

\begin{proof}
     We first prove that $\underline{X}\cap B(O,n)$ is independent of the embedding in the affine space. Consider $\mathcal{L}(X)$ the scheme of germs of arcs on $X$ (see \cite{denef_germs_1999} for more details about arc spaces). We denote by $\pi_n : \mathcal{L}(X) \longrightarrow \mathcal{L}_n(X)$ the canonical morphism corresponding to truncation of arcs. Hence, we can write $\underline{X}\cap B(O,n)$ as $\pi_{n-1}^{-1}(\mathcal{L}_{n-1}(X)_0) \subset \mathcal{L}(X)$, where $\mathcal{L}_{n-1}(X)_0$ are the $(n-1)-$jets that specialize to $O$. Clearly this do not depend of the choice of an embedding. It remains to prove the independence of the volume forms. Let $i_1 : \underline{X} \longrightarrow h[N,0,0]$ and $i_2 : \underline{X} \longrightarrow h[N',0,0]$ be two embeddings. Consider the diagonal embedding $i_3=(i_1,i_2) : \underline{X} \longrightarrow h[N+N',0,0]$. We get the canonical volume forms $|\omega_0|_{i_1(\underline{X})}, |\omega_0|_{i_2(\underline{X})}, |\omega_0|_{i_3(\underline{X})}$. Let $\pi_1 : i_3(\underline{X}) \longrightarrow i_1(\underline{X})$, $\pi_2 : i_3(\underline{X}) \longrightarrow i_2(\underline{X})$ be the coordinates projections. Hence, $i_3(\underline{X})$ is the graph of the definable bijection $\phi = \pi_2 \circ \pi_1^{-1}$, whose Jacobian determinant is equal to 1. Hence, $(\pi_2\circ\pi_1^{-1})^*|\omega_0|_{i_2(\underline{X})}=|\omega_0|_{i_1(\underline{X})}$, so the motivic volume is independent of the choice of an embedding. 
     
     %Hence, the volume form $\omega$ is the canonical volume form on $X(\C(\!(t)\!))$. Since the definitions \ref{def : volume form embed} and \ref{def : volume form modele} coincide (where $\mathcal{X}^0$ is $X$) and that the volume form in the definition \ref{def : volume form modele} is independent of the embedding, hence so is $\omega$.
\end{proof}

\vspace{\baselineskip}

We end this section by recalling the following definitions of \cite{cluckers_local_2012}, in order to prove that the motivic local density is a 1-bilipschitz invariant for the non-Archimedean norm.

\vspace{\baselineskip}

\begin{definition}~
    Let $x$ be an element of $k(\!(t)\!)$ and $|x|:=\rho^{v(x)}$ with $0<\rho<1$. 
    \begin{enumerate}
        \item For $\varepsilon$ a positive real number, $D$ a subset of $k(\!(t)\!)^n$ and $f : D \longrightarrow k(\!(t)\!)^m$, then $f$ is $\varepsilon$-Lipschitz if for all $x,y \in D$ one has:
        $$
        |f(x)-f(y)| \leq \varepsilon |x-y|.
        $$
        \item  A function $f : k(\!(t)\!)^n \longrightarrow k(\!(t)\!)^m$ is a bilipschitz homeomorphism if $f$ is a bijection and there exists a real constant $C \geq 1$ such that for all $x,x' \in k(\!(t)\!)^n$, one has:
        $$
        \frac{1}{C}|x-x'|\leq |f(x)-f(x')|\leq C|x-x'|.
        $$
        \item We say that a definable analytic function $f : U \longrightarrow  k(\!(t)\!)^m$ with $U$ an open definable subset of $k(\!(t)\!)^n$ is $\varepsilon$-analytic if the norm $|Df|=\text{max}_{i,j}  \lvert \frac{\partial f_j}{\partial x_i}\rvert$ of the differential of $f$ is less or equal than $\varepsilon$ at every point of $U$.
    \end{enumerate}  
\end{definition}

\vspace{\baselineskip}

\begin{proposition}
    Let $U$ be a definable open in $k(\!(t)\!)^n$ and let $f : U\longrightarrow k(\!(t)\!)^m$ be a definable analytic function which is locally $\varepsilon-$Lipschitz. Then $f$ is $\varepsilon-$analytic.
\end{proposition}

\vspace{\baselineskip}

\begin{proof}
    Suppose by contradiction that $|Df| > \varepsilon$ at $x \in U$. Without loss of generality, we can choose suppose $\varepsilon =1$, $x=0$ and choose as a neighborhood of $x$ in $U$ the unit ball such that the component functions $f_i$ of $f$ are given by converging power series on it. By assumption, we have for some $i,j$ that $\lvert\frac{\partial f_j(x)}{x_i}(0) \rvert>1$, hence the linear term of $f_j$ in $x_i$ has a coefficient of norm strictly bigger than 1. By the convergence of the power series, the coefficients of the $f_j$ are bounded in norm, say by $C$. Hence, for any $x$ in $ \{x \in B(0,0) \mid |x| < \frac{1}{C}\}$, we have that $|f(x)-f(0)|>|x|$ which contradicts the fact $f$ is $\varepsilon$-Lipschitz with $\varepsilon=1$.
\end{proof}

\vspace{\baselineskip}

\begin{proposition}
    The motivic local density is a 1-bilipschitz invariant for the non-Archimedean norm.
\end{proposition}

\vspace{\baselineskip}

\begin{proof}
    Let $f : \underline{X} \longrightarrow \underline{Y}$ be a 1-bilipschitz definable bijection, and let $d=\dim(\underline{X})=\dim(\underline{Y})$. Hence for all $x,x' \in \underline{X}$, we have $|x-x'|=|f(x)-f(x')|$. For $y=f(x)$, we have that $\underline{Y}\cap B(y,n)=f(B(x,n)\cap \underline{X})$, hence $\mu_d(\underline{Y}\cap B(y,n))=\mu_d(f(B(x,n)\cap \underline{X}))$. It remains to prove that $\mu_d(f(B(x,n)\cap \underline{X}))=\mu_d(\underline{X}\cap B(x,n))$. Since $f$ is a bijection, we can apply the theorem of change of variables:
    $$
    \mu_d(f(B(x,n)\cap \underline{X}))=\int_{\underline{X}\cap B(x,n)} \LL^{-v(\text{Jac}(f))}.
    $$

    Since $f$ is 1-Lipschitz, hence it is 1-analytic and $|Df| \leq 1$, and $f^{-1}$ is 1-Lipschitz, hence we conclude that $|Df|=1$. We finally get that $\text{det}(Df), \text{det}(Df^{-1}) \in \mathcal{O}_{k(\!(t)\!)}$, so $\text{det}(Df) \in \mathcal{O}_{k(\!(t)\!)}^\times$ and $v(\text{det}(f))=0$.

    Finally, $\mu_d(f(B(x,n)\cap \underline{X}))=\mu_d(B(x,n)\cap \underline{X})$, and then $\Theta^\text{mot}_d(X,x)=\Theta^\text{mot}_d(Y,y)$. 
\end{proof}

\section{Proof of the formula for surfaces}

    In this section, we prove Theorem \ref{thm : formula surface}. We recall the following: let $\pi : Y \longrightarrow X$ be a resolution of the singularity that factors through the blowup of the maximal ideal and the Nash transform of $(X,O)$, and such that the exceptional locus is simple normal crossing. We denote by $E_i$ the irreducible component of the exceptional divisor $D=\pi^{-1}(0)$.

\vspace{\baselineskip}

\begin{theorem}
    Let $X$ be an algebraic surface (i.e. a 2-dimensional separated scheme of finite type) over $\C$ with a unique isolated singularity $O$. Let $\pi : Y \longrightarrow X$ be a good resolution of the singularity that satisfies the previous conditions. Denote $E_i$ the irreducible components of the exceptional locus, $m_i$ its multiplicity and $q_i$ its associated inner rate. Set $E_i^0=E_i \smallsetminus \cup_{j \neq i} E_j$.
    The motivic local density of $X$ at the point $O$ is given by:
    $$
    \Theta^\text{mot}_2(X,O) = \sum_{i\mid q_i=1} \frac{1}{m_i} \Bigg(\frac{1}{\LL+1}[E_i^0] + \sum_{j \mid E_i\cap E_j \neq \emptyset }\frac{\LL^{-(q_{j}-1)m_{j}}(1-\LL^{-1})}{(1-\LL^{-(q_{j}-1)m_{j}})(1+\LL^{-1})} \Bigg),
    $$
    which lies in $\mathcal{M}_\C \otimes \Q$.
\end{theorem}

\vspace{\baselineskip}

\begin{remark}
    Consider $E_i$ an exceptional component of the exceptional locus such that $q_i$ its associated inner rate is equal to one. If $E_i \cap E_j \neq 0$ for $E_j$ an other exceptional component, then $q_j \neq 1$. 
\end{remark}

\vspace{\baselineskip}

\begin{lemma}\label{lem:vois tub}
Let $X$ be a surface with a unique isolated singularity in $O$, and consider a resolution $\pi : Y \longrightarrow X$ such that the exceptional locus is simple normal crossing. Then $\pi^{-1}(\underline{X} \cap S(O,n)) = \mathcal{V}(D)$ where $\mathcal{V}(D)$ is the set of points of valuative distance exactly $\frac{n}{m_i}$ of every component $E_i$ of the exceptional divisor $D=\pi^{-1}(O)=\sum_{i\in I} m_iE_i$, and $S(O,n)$ is the valuative sphere. 
\end{lemma}

\vspace{\baselineskip}

\begin{proof}
Let us set $D=\pi^{-1}(O)=\sum_{i\in I} m_iE_i$ the exceptional divisor, with $m_i$ the multiplicity of the exceptional component $E_i$. Let $E_i^0 = E_i \smallsetminus \cup_{i\neq j} E_j$. Since the divisor is supposed to be normal crossing, we can cover $Y$ by affine open subsets $U$ on which there exist regular functions $x,y$ inducing an étale map $U \longrightarrow \mathbb{A}_k^2$ such that on $U$, each component of the exceptional divisor is defined by a monomial in $x,y$.

In each point of $E_i^0$, we have:
\begin{align*}
    \mathcal{V}(D)\cap \underline{U}&=\{p\in \underline{U} \mid v(x(p)^{m_i})=n\} \text{ and} \\
    \pi^{-1}(\underline{X} \cap S(O,n))\cap \underline{U} &= \{p\in \underline{U} \mid \pi(p) \in S(O,n)\} \\
    & = \{p\in \underline{U} \mid v(x(p)^{m_i})=n\}.
\end{align*}

In each point of $E_i \cap E_j$, we have:
\begin{align*}
    \mathcal{V}(D)\cap \underline{U}&=\{p\in \underline{U} \mid v(x^{m_i}y^{m_j}(p))=n\} \text{ and} \\
    \pi^{-1}(\underline{X} \cap S(O,n))\cap \underline{U}&= \{p\in \underline{U} \mid \pi(p) \in S(O,n)\} \\
    & = \{p\in \underline{U} \mid v(x^{m_i}y^{m_j}(p))=n\}.
\end{align*}

\end{proof}

\vspace{\baselineskip}

\begin{proof}[Proof of Theorem \ref{thm : formula surface}]
    Since $\pi$ induces an isomorphism $Y \smallsetminus \pi^{-1}(O) \longrightarrow X \smallsetminus \{O\}$, we have the following equality:
    \begin{align*}
        \mu_2(\underline{X} \cap S(O,n)) & = \mu_2(\underline{X}\smallsetminus\{O\} \cap S(O,n)) \\
        & = \int_{\pi^{-1}(\underline{X}\smallsetminus\{O\} \cap S(O,n))} \LL^{-\text{ord}(\pi^*(\Omega_X^2))}
    \end{align*}
    
    using the theorem of change of variables \ref{thm: change var}. Using Lemma \ref{lem:vois tub}, we then have:
    $$
    \mu_2(\underline{X}\smallsetminus\{O\} \cap S(O,n)) = \int_{\mathcal{V}(D)} \LL^{-\text{ord}(\pi^*(\Omega_X^2))}.
    $$
    As in the proof of the previous lemma, in each point of the support of the divisor, we can define a finite Zariski cover as follow.  In a point of $E_i^0$, working locally on $U$ an open set of $Y$ we can assume that $E_i$ is defined by the equation $x=0$, and in a point of $E_i \cap E_j$, on an open set of $Y$ we can assume that $E_i \cap E_j$ is defined by the equation $xy=0$. By additivity of the motivic measure, it is sufficient to work on those local neighborhoods. By definition (See Section \ref{section : mather}), the function $\text{ord}(\pi^*(\Omega_X^2))$ is equal to the Mather discrepancy function, and thanks to the Proposition \ref{disc mather}, we can compute this function using the multiplicities and the inner rates associated to the component of the exceptional divisor. Hence, near each divisor, the power of $\LL$ appearing in the integral is very easy to explicit. 

    \vspace{\baselineskip}

    We first work in a neighborhood $U$ intersecting $E_i^0$: for $n$ such that $m_i$ divides $n$, we get:
    \begin{align*}
        \int_{v(x^{m_i})=n} \LL^{-\hat{k}_iv(x)} & = \LL^{-((q_i+1)m_i-1)\frac{n}{m_i}}\LL^{-\frac{n}{m_i}}\LL^{-2}(\LL-1)[E_i^0\cap U]\\
        & = \LL^{-(q_i+1)n}\LL^{-2}(\LL-1)[E_i^0\cap U].
    \end{align*}
    For $q_i >1$, the limit after normalizing by the volume of the sphere $\LL^{-2n}(1-\LL^{-2})$ will be zero. Hence, we only focus on divisors whose associated inner rate is equal to 1. 

    We now need to consider a neighborhood $U$ intersecting $E_i \cap E_j$. We fix some notations for the rest of the proof. Up to reordering, we can assume that $E_1$ is an exceptional divisor whose associated inner rate is equal to 1, and multiplicity is $m_1$, and $E_2$ is an exceptional divisor intersecting $E_1$ transversely whose inner rate is denoted $q_2 >1$, and multiplicity $m_2$. We want to compute:
    \begin{align*}
    & \int_{v(x^{m_1}y^{m_2})=n} \LL^{-\hat{k}_xv(x)-\hat{k}_yv(y)} \\
 = & \int_{v(x^{m_1})=n,v(y)=0} \LL^{-\hat{k}_xv(x)} + \int_{v(y^{m_2})=n,v(x)=0} \LL^{-\hat{k}_yv(y)} \\
     + & \int_{v(x^{m_1}y^{m_2})=n, v(x),v(y)>0} \LL^{-\hat{k}_xv(x)-\hat{k}_yv(y)},
\end{align*}

where $\hat{k}_x=2m_1-1, \hat{k}_y=(q_2+1)m_2-1$. For $m_1 \mid n$, 
\begin{align*}
    &\int_{v(x^{m_1})=n,v(y)=0} \LL^{-\hat{k}_xv(x)} \\
    =&\LL^{-(2m_1-1)\frac{n}{m}}\mu_2(\{(x,y) \in U \mid v(x)=\frac{n}{m_1}, v(y)=0\}) \\
    =& \LL^{-(2m_1-1)\frac{n}{m_1}} \LL^{-\frac{n}{m_1}-1}(\LL-1)\LL^{-1}[E_1^0\cap U] \\
    =& \LL^{-2n}\LL^{-2}(\LL-1)[E_1^0\cap U].
\end{align*}

As seen before, since $q_2>1$, the mean value at infinity will be zero after normalization for the second integral. We now focus on the third integral:
\begin{align}
    \int_{v(x^{m_1}y^{m_2})=n, v(x),v(y)>0} \LL^{-\hat{k}_xv(x)-\hat{k}_yv(y)} \notag \\
    = \sum_{m_1k+m_2l=n,k,l\geq 1} \LL^{-(\hat{k}_x+1)k}\LL^{-(\hat{k}_y+1)l}\LL^{-2}(\LL-1)^2 .\tag{1} \label{eq : integral intersection}
\end{align}

Consider for simplicity gcd$(m_1,m_2)$=1 at first. The solutions of the equation $m_1k+m_2l=n$ are of the form $nk_0-m_2\alpha=k, nl_0+m_1\alpha = l$ for $k_0,l_0$ fixed integers. We set $a=2m_1, b=(q_2+1)m_2, t_1 = \lceil\frac{1-nl_0}{m_1}\rceil$ (the ceiling function) and $t_2=\lfloor \frac{1-nk_0}{-m_2}\rfloor$ (the floor function). We write (\ref{eq : integral intersection}) as: 
\begin{align*}
    &\LL^{-n(ak_0+bl_0)}\LL^{-2}(\LL-1)^2\sum_{\alpha = t_1}^{t_2} \LL^{-m_1\alpha(b-2m_2)} \\
    =&\LL^{-n(ak_0+bl_0)}\LL^{-2}(\LL-1)^2\frac{\LL^{-m_1(b-2m_2)t_1}-\LL^{-m_1(b-2m_2)(t_2+1)}}{1-\LL^{-m_1(b-2m_2)}}.
\end{align*}

Again, we can see that the terms depending on $t_2$ will not contribute after normalization and taking the limit. We will ignore it in the rest of the proof. Now set $n=m_1s+e, 0\leq e \leq m_1-1$. Our integral (\ref{eq : integral intersection}) becomes:
$$
\LL^{-2n -(l_0e+\lceil\frac{1-l_0e}{m_1}\rceil m_1)(b-2m_1)} \frac{\LL^{-2}(\LL-1)^2}{1-\LL^{-m_1(b-2m_2)}}.
$$
By taking the mean value at infinity, we will have to compute:
$$
\sum_{e=0}^{m_1-1} \LL^{-(\frac{l_0e}{m_1}+\lceil\frac{1-l_0e}{m_1}\rceil)m_1(b-2m_2)}.
$$

For $1 \leq e \leq m_1-1$, we have that $\lceil\frac{1-l_0e}{m_1}\rceil = - \lfloor \frac{l_0e}{m_1} \rfloor$, hence $\frac{l_0e}{m_1}+\lceil\frac{1-l_0e}{m_1}\rceil = \Big\{\frac{l_0e}{m_1}\Big\}$ the fractional part. Since $l_0$ and $m_1$ are coprime, for $1 \leq e \leq m_1-1$, $\Big\{\frac{l_0e}{m_1}\Big\}m_1$ takes all values between 1 and $m_1-1$. We can write the previous sum as:

\begin{align*}
    \LL^{-m_1(b-2m_2)}+ \sum_{e=1}^{m_1-1}\LL^{-\Big\{\frac{l_0e}{m_1}\Big\}m_1(b-2m_2)} 
    & = \LL^{-m_1(b-2m_2)} + \sum_{\beta=1}^{m_1-1}\LL^{-\beta(b-2m_2)} \\
    &= \LL^{-(b-2m_2)}\frac{1-\LL^{-m_1(b-2m_2)}}{1-\LL^{-(b-2m_2)}}.
\end{align*}

Finally, after normalizing by $\LL^{-2n}(1-\LL^{-2})$, taking the mean value at infinity, we have the local density:
$$
\frac{1}{m_1} \Big[\frac{\LL^{-1}}{1+\LL^{-1}}[E_1^0]+ \frac{\LL^{-(b-2m_2)}(1-\LL^{-1})}{(1-\LL^{-(b-2m_2)})(1+\LL^{-1})} \Big]
$$
with $b-2m_2=(q_2-1)m_2$. 

\vspace{\baselineskip}

Now consider the case where gcd$(m_1,m_2)=d \neq 1$. Then the equation $m_1k+m_2l=n$ has solutions if and only if $d$ divides $n$. We set $m_1'=\frac{m_1}{d}, m_2'=\frac{m_2}{d}, n'=\frac{n}{d}$. We now have to compute:
$$
\sum_{m_1'k+m_2'l=n',k,l\geq 1} \LL^{-(\hat{k}_x+1)k}\LL^{-(\hat{k}_y+1)l}\LL^{-2}(\LL-1)^2.
$$
Since gcd$(m_1',m_2')=1$, we can do the same computation as above. Hence, we finally get a term:
$$
\frac{1}{m'_1} \Big[\frac{\LL^{-1}}{1+\LL^{-1}}[E_1^0]+ \frac{\LL^{-(b-2m_2)}(1-\LL^{-1})}{(1-\LL^{-(b-2m_2)})(1+\LL^{-1})} \Big]
$$
and with the extra condition that $d$ divides $n$, the local density will be given by:
$$
\frac{1}{d}\frac{1}{m'_1} \Big[\frac{\LL^{-1}}{1+\LL^{-1}}[E_1^0]+ \frac{\LL^{-(b-2m_2)}(1-\LL^{-1})}{(1-\LL^{-(b-2m_2)})(1+\LL^{-1})} \Big].
$$

\end{proof}

\begin{example}
    We consider the singularity $E_8$. Let $f(x,y,z)=x^2+y^3+z^5$ and $X:=\{(x,y,z) \in \C \mid f(x,y,z)=0\}$. The Newton polyhedra is defined as $\mathcal{N}(f)=\text{Conv}(\text{Supp}(f)+\N^3)$, and $\mathcal{N}_0(f)$ is the set of compact faces of $\mathcal{N}(f)$. Let $\underline{X}_{\Delta_1}$ be the subdefinable set of $\underline{X}$ with the additional condition $2v(x)=3v(y)$. Since 
    $$
    \underline{X}\cap B(0,n) = \bigsqcup_{\Delta \in \mathcal{N}_0(f)} \underline{X}_\Delta\cap B(0,n)
    $$
    by additivity of the local density, it suffices to compute the density of $\underline{X}_\Delta$ and sum all the results. One of the argument for the computations is Hensel's lemma in order to write our sets as balls parametrized by the residue field. This is possible since $f$ is non-degenerate on every face of the Newton polyhedra. The condition $2v(x)=3v(y)$ forces $v(y)$ to be even, hence we will get two radius of ball to consider: even and odd. The local density $\Theta_2^\text{mot}(X_{\Delta_1},0)$ is equal to $\frac{1}{2}$, and one can show that the other faces will not contribute to the local density. Hence,
    $$
    \Theta_2^\text{mot}(X,0) = \frac{1}{2}.
    $$

    Consider $\Gamma_\pi$ the graph of resolution of $E_8$, where the vertices are decorated with the multiplicities of the corresponding exceptional components (in parenthesis) and the associated inner rate, as computed in example 3.7 of \cite{silva_inner_2022}. \\
    
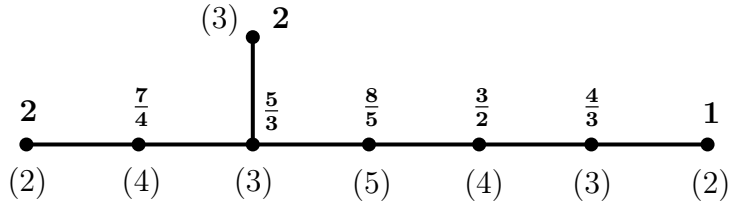
    \begin{center}
        \tikzset{every picture/.style={line width=0.75pt}}
        \begin{tikzpicture}[x=0.75pt,y=0.75pt,yscale=-1,xscale=1]
        \draw [line width=1.5]    (138,123) -- (480,123) ;
        \draw  [fill={rgb, 255:red, 0; green, 0; blue, 0 }  ,fill opacity=1 ] (144,123) .. controls (144,121.34) and (142.66,120) .. (141,120) .. controls (139.34,120) and (138,121.34) .. (138,123) .. controls (138,124.66) and (139.34,126) .. (141,126) .. controls (142.66,126) and (144,124.66) .. (144,123) -- cycle ;
        \draw  [fill={rgb, 255:red, 0; green, 0; blue, 0 }  ,fill opacity=1 ] (200,123) .. controls (200,121.34) and (198.66,120) .. (197,120) .. controls (195.34,120) and (194,121.34) .. (194,123) .. controls (194,124.66) and (195.34,126) .. (197,126) .. controls (198.66,126) and (200,124.66) .. (200,123) -- cycle ;
        \draw  [fill={rgb, 255:red, 0; green, 0; blue, 0 }  ,fill opacity=1 ] (426,123) .. controls (426,121.34) and (424.66,120) .. (423,120) .. controls (421.34,120) and (420,121.34) .. (420,123) .. controls (420,124.66) and (421.34,126) .. (423,126) .. controls (424.66,126) and (426,124.66) .. (426,123) -- cycle ;
        \draw  [fill={rgb, 255:red, 0; green, 0; blue, 0 }  ,fill opacity=1 ] (484,123) .. controls (484,121.34) and (482.66,120) .. (481,120) .. controls (479.34,120) and (478,121.34) .. (478,123) .. controls (478,124.66) and (479.34,126) .. (481,126) .. controls (482.66,126) and (484,124.66) .. (484,123) -- cycle ;
        \draw  [fill={rgb, 255:red, 0; green, 0; blue, 0 }  ,fill opacity=1 ] (257,123) .. controls (257,121.34) and (255.66,120) .. (254,120) .. controls (252.34,120) and (251,121.34) .. (251,123) .. controls (251,124.66) and (252.34,126) .. (254,126) .. controls (255.66,126) and (257,124.66) .. (257,123) -- cycle ;
        \draw  [fill={rgb, 255:red, 0; green, 0; blue, 0 }  ,fill opacity=1 ] (315,123) .. controls (315,121.34) and (313.66,120) .. (312,120) .. controls (310.34,120) and (309,121.34) .. (309,123) .. controls (309,124.66) and (310.34,126) .. (312,126) .. controls (313.66,126) and (315,124.66) .. (315,123) -- cycle ;
        \draw  [fill={rgb, 255:red, 0; green, 0; blue, 0 }  ,fill opacity=1 ] (370,123) .. controls (370,121.34) and (368.66,120) .. (367,120) .. controls (365.34,120) and (364,121.34) .. (364,123) .. controls (364,124.66) and (365.34,126) .. (367,126) .. controls (368.66,126) and (370,124.66) .. (370,123) -- cycle ;
        \draw  [fill={rgb, 255:red, 0; green, 0; blue, 0 }  ,fill opacity=1 ] (257,69) .. controls (257,67.34) and (255.66,66) .. (254,66) .. controls (252.34,66) and (251,67.34) .. (251,69) .. controls (251,70.66) and (252.34,72) .. (254,72) .. controls (255.66,72) and (257,70.66) .. (257,69) -- cycle ;
        \draw [line width=1.5]    (254,72) -- (254,123) ;
        \draw (130,133) node [anchor=north west][inner sep=0.75pt]   [align=left] {(2)};
        \draw (187,133) node [anchor=north west][inner sep=0.75pt]   [align=left] {(4)};
        \draw (243,133) node [anchor=north west][inner sep=0.75pt]   [align=left] {(3)};
        \draw (302,134) node [anchor=north west][inner sep=0.75pt]   [align=left] {(5)};
        \draw (358,134) node [anchor=north west][inner sep=0.75pt]   [align=left] {(4)};
        \draw (412,134) node [anchor=north west][inner sep=0.75pt]   [align=left] {(3)};
        \draw (471,134) node [anchor=north west][inner sep=0.75pt]   [align=left] {(2)};
        \draw (226,51) node [anchor=north west][inner sep=0.75pt]   [align=left] {(3)};
        \draw (262,52.4) node [anchor=north west][inner sep=0.75pt]    {$\mathbf{2}$};
        \draw (477,100.4) node [anchor=north west][inner sep=0.75pt]    {$\mathbf{1}$};
        \draw (417,93.4) node [anchor=north west][inner sep=0.75pt]    {$\mathbf{{\textstyle \frac{4}{3}}}$};
        \draw (362,93.4) node [anchor=north west][inner sep=0.75pt]    {$\mathbf{{\textstyle \frac{3}{2}}}$};
        \draw (307,93.4) node [anchor=north west][inner sep=0.75pt]    {$\mathbf{{\textstyle \frac{8}{5}}}$};
        \draw (257,95.4) node [anchor=north west][inner sep=0.75pt]    {$\mathbf{{\textstyle \frac{5}{3}}}$};
        \draw (192,93.4) node [anchor=north west][inner sep=0.75pt]    {$\mathbf{{\textstyle \frac{7}{4}}}$};
        \draw (136,99.4) node [anchor=north west][inner sep=0.75pt]    {$\mathbf{2}$};
        \end{tikzpicture}
    \end{center}

    Using the formula of Theorem \ref{thm : formula surface}, we find :
    $$
    \Theta_2^\text{mot}(X,0)=\frac{1}{2} \Big[\frac{\LL^{-1}}{1+\LL^{-1}}\LL+ \frac{\LL^{-1}(1-\LL^{-1})}{(1-\LL^{-1})(1+\LL^{-1})} \Big] = \frac{1}{2}.
    $$
\end{example}

\vspace{\baselineskip}

\begin{example}
    As studied in example 5.8 of \cite{silva_inner_2022}, consider the hypersurface singularity $(X,0) \subset (\C^3,0)$ defined by $(zx^2+y^3)(x^3+zy^2)+z^7=0$. After computing a good resolution of $(X,0)$, we get the following dual graph, where the vertices are decorated with the multiplicities of the corresponding exceptional components (in parenthesis) and the associated inner rate.

    \begin{center}
        \tikzset{every picture/.style={line width=0.75pt}}        
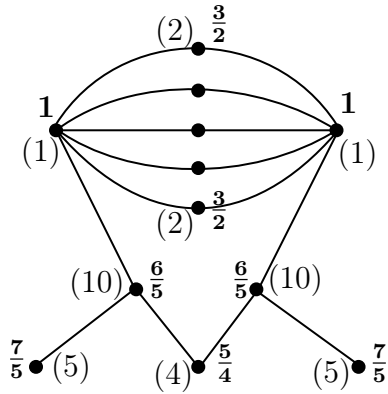
        \begin{tikzpicture}[x=0.75pt,y=0.75pt,yscale=-1,xscale=1]
        \draw  [fill={rgb, 255:red, 0; green, 0; blue, 0 }  ,fill opacity=1 ] (163,141) .. controls (163,139.34) and (161.66,138) .. (160,138) .. controls (158.34,138) and (157,139.34) .. (157,141) .. controls (157,142.66) and (158.34,144) .. (160,144) .. controls (161.66,144) and (163,142.66) .. (163,141) -- cycle ;
        \draw  [fill={rgb, 255:red, 0; green, 0; blue, 0 }  ,fill opacity=1 ] (303,141) .. controls (303,139.34) and (301.66,138) .. (300,138) .. controls (298.34,138) and (297,139.34) .. (297,141) .. controls (297,142.66) and (298.34,144) .. (300,144) .. controls (301.66,144) and (303,142.66) .. (303,141) -- cycle ;
        \draw  [fill={rgb, 255:red, 0; green, 0; blue, 0 }  ,fill opacity=1 ] (234,141) .. controls (234,139.34) and (232.66,138) .. (231,138) .. controls (229.34,138) and (228,139.34) .. (228,141) .. controls (228,142.66) and (229.34,144) .. (231,144) .. controls (232.66,144) and (234,142.66) .. (234,141) -- cycle ;
        \draw  [fill={rgb, 255:red, 0; green, 0; blue, 0 }  ,fill opacity=1 ] (234,121) .. controls (234,119.34) and (232.66,118) .. (231,118) .. controls (229.34,118) and (228,119.34) .. (228,121) .. controls (228,122.66) and (229.34,124) .. (231,124) .. controls (232.66,124) and (234,122.66) .. (234,121) -- cycle ;
        \draw  [fill={rgb, 255:red, 0; green, 0; blue, 0 }  ,fill opacity=1 ] (234,100) .. controls (234,98.34) and (232.66,97) .. (231,97) .. controls (229.34,97) and (228,98.34) .. (228,100) .. controls (228,101.66) and (229.34,103) .. (231,103) .. controls (232.66,103) and (234,101.66) .. (234,100) -- cycle ;
        \draw  [fill={rgb, 255:red, 0; green, 0; blue, 0 }  ,fill opacity=1 ] (234,160) .. controls (234,158.34) and (232.66,157) .. (231,157) .. controls (229.34,157) and (228,158.34) .. (228,160) .. controls (228,161.66) and (229.34,163) .. (231,163) .. controls (232.66,163) and (234,161.66) .. (234,160) -- cycle ;
        \draw  [fill={rgb, 255:red, 0; green, 0; blue, 0 }  ,fill opacity=1 ] (234,180) .. controls (234,178.34) and (232.66,177) .. (231,177) .. controls (229.34,177) and (228,178.34) .. (228,180) .. controls (228,181.66) and (229.34,183) .. (231,183) .. controls (232.66,183) and (234,181.66) .. (234,180) -- cycle ;
        \draw  [fill={rgb, 255:red, 0; green, 0; blue, 0 }  ,fill opacity=1 ] (203,221) .. controls (203,219.34) and (201.66,218) .. (200,218) .. controls (198.34,218) and (197,219.34) .. (197,221) .. controls (197,222.66) and (198.34,224) .. (200,224) .. controls (201.66,224) and (203,222.66) .. (203,221) -- cycle ;
        \draw  [fill={rgb, 255:red, 0; green, 0; blue, 0 }  ,fill opacity=1 ] (234,260) .. controls (234,258.34) and (232.66,257) .. (231,257) .. controls (229.34,257) and (228,258.34) .. (228,260) .. controls (228,261.66) and (229.34,263) .. (231,263) .. controls (232.66,263) and (234,261.66) .. (234,260) -- cycle ;
        \draw  [fill={rgb, 255:red, 0; green, 0; blue, 0 }  ,fill opacity=1 ] (263,221) .. controls (263,219.34) and (261.66,218) .. (260,218) .. controls (258.34,218) and (257,219.34) .. (257,221) .. controls (257,222.66) and (258.34,224) .. (260,224) .. controls (261.66,224) and (263,222.66) .. (263,221) -- cycle ;
        \draw  [fill={rgb, 255:red, 0; green, 0; blue, 0 }  ,fill opacity=1 ] (153,260) .. controls (153,258.34) and (151.66,257) .. (150,257) .. controls (148.34,257) and (147,258.34) .. (147,260) .. controls (147,261.66) and (148.34,263) .. (150,263) .. controls (151.66,263) and (153,261.66) .. (153,260) -- cycle ;
        \draw  [fill={rgb, 255:red, 0; green, 0; blue, 0 }  ,fill opacity=1 ] (314,260) .. controls (314,258.34) and (312.66,257) .. (311,257) .. controls (309.34,257) and (308,258.34) .. (308,260) .. controls (308,261.66) and (309.34,263) .. (311,263) .. controls (312.66,263) and (314,261.66) .. (314,260) -- cycle ;
        \draw    (160,141) -- (200,222) ; 
        \draw    (301,141) -- (261,221) ;
        \draw    (149,260) -- (199,221) ;
        \draw    (259,221) -- (311,259) ;
        \draw    (231,260) -- (200,221) ;
        \draw    (261,221) -- (231,261) ;
        \draw    (160,141) .. controls (200,111) and (262,116) .. (300,141) ;
        \draw    (160,138) .. controls (192,87) and (267,87) .. (300,138) ;
        \draw    (160,141) -- (297,141) ;
        \draw    (161,141) .. controls (196,166) and (257,168) .. (300,141) ;
        \draw    (161,142) .. controls (206,194) and (257,192) .. (300,142) ;
        \draw (235,76.4) node [anchor=north west][inner sep=0.75pt]    {$\mathbf{{\textstyle \frac{3}{2}}}$};
        \draw (235,170.4) node [anchor=north west][inner sep=0.75pt]    {$\mathbf{{\textstyle \frac{3}{2}}}$};
        \draw (134,245.4) node [anchor=north west][inner sep=0.75pt]    {$\mathbf{{\textstyle \frac{7}{5}}}$};
        \draw (315,247.4) node [anchor=north west][inner sep=0.75pt]    {$\mathbf{{\textstyle \frac{7}{5}}}$};
        \draw (238,247.4) node [anchor=north west][inner sep=0.75pt]    {$\mathbf{{\textstyle \frac{5}{4}}}$};
        \draw (149,122.4) node [anchor=north west][inner sep=0.75pt]    {$\mathbf{1}$};
        \draw (300,120.4) node [anchor=north west][inner sep=0.75pt]    {$\mathbf{1}$};
        \draw (246,205.4) node [anchor=north west][inner sep=0.75pt]    {$\mathbf{{\textstyle \frac{6}{5}}}$};
        \draw (204,205.4) node [anchor=north west][inner sep=0.75pt]    {$\mathbf{{\textstyle \frac{6}{5}}}$};
        \draw (141,143) node [anchor=north west][inner sep=0.75pt]   [align=left] {(1)};
        \draw (299,144) node [anchor=north west][inner sep=0.75pt]   [align=left] {(1)};
        \draw (208,82) node [anchor=north west][inner sep=0.75pt]   [align=left] {(2)};
        \draw (208,179) node [anchor=north west][inner sep=0.75pt]   [align=left] {(2)};
        \draw (165,209) node [anchor=north west][inner sep=0.75pt]   [align=left] {(10)};
        \draw (264,207) node [anchor=north west][inner sep=0.75pt]   [align=left] {(10)};
        \draw (156,251) node [anchor=north west][inner sep=0.75pt]   [align=left] {(5)};
        \draw (286,254) node [anchor=north west][inner sep=0.75pt]   [align=left] {(5)};
        \draw (207,255) node [anchor=north west][inner sep=0.75pt]   [align=left] {(4)};
        \end{tikzpicture}
    \end{center}
    Using the formula of Theorem \ref{thm : formula surface}, we get that the motivic local density of this surface is given by:
    \begin{align*}
        &2\Big(\frac{\LL-5}{\LL+1} +5\frac{\LL^{-1}(1-\LL^{-1})}{(1-\LL^{-1})(1+\LL^{-1})} +\frac{\LL^{-2}(1-\LL^{-1})}{(1-\LL^{-2})(1+\LL^{-1})}\Big) \\
        =&2\Big(\frac{\LL}{\LL+1}+\frac{1}{(\LL+1)^2}\Big).
    \end{align*}
\end{example}

\vspace{\baselineskip}

We will now give the proof of Proposition \ref{prop : courbes}, using the same technique as before. Let us state the proposition:

\vspace{\baselineskip}

\begin{proposition}\label{prop : curve section 6}
    Let $f \in \C[x,y]$ be a power series without square factor. Let $f=f_1\dots f_r$ the decomposition in $\mathbb{C}[\![x,y]\!]$ of $f$ into irreducible factors. Denote by $N_i$ the multiplicity of $f_i$ and $C$ the curve defined by $f$. Then the motivic density of $C$ at the origin is given by:
    $$
    \Theta^\text{mot}_1(C,0)=\sum_{i=1}^r \frac{1}{N_i}.
    $$
\end{proposition}

\vspace{\baselineskip}

\begin{proof}
    As in the case of surface, we need to compute the order of the jacobian of the resolution morphism in order to use the change of variable formula. We will use the same tools as in \cite{cherik_inner_2022} to compute the canonical divisor associated to the following 1-form: consider $l : C \longrightarrow \C$ a generic projection and set $\omega:=l^*dx \in \Omega_C$. Let $\pi : Y \longrightarrow C$ a good resolution of the singularity of $C$, we denote by $D=\pi^{-1}(0)$ the exceptional locus which is a finite set of points denoted $p_i$ of multiplicity $N'_i$ for $i=1,\dots,r$. We first prove that $N_i'=N_i$. By the Newton-Puiseux theorem, each branch $C_i=\{f_i=0\}$ admits a parametrization $\pi_i:(\C,0)\longrightarrow C, \pi_i(t)=(x_i(t),\phi_i(t))$ with $x_i(t)=t^{N_i}, \phi_i(t)=\sum_{l \geq N_i}a_lt^l$. Consider the global normalization map $\pi : Y \longrightarrow C$ and let $\pi^{-1}(0)=\{p_1,\dots,p_r\}$. Let $t$ be a local coordinate on $Y$ such that $p_i=\{t_i=0\}$, we have by the property of normalization that $\pi(t)=(t^{N_i},\phi_i(t))$, hence $\text{deg}(p_i)=N_i'=N_i$. See also \cite{wall_singular_2004} for results on resolution of plane curve singularity, especially section 3.
    
    Consider $p$ a smooth point of $Y \smallsetminus (\pi^{-1}(0) \cup l^*)$ for $l^*$ the strict transfrom of $l$. Let $u$ be a system of local coordinates of $Y$ centered at $p$ such that $p_i$ has local equation $u=0$ and such that $(l \circ \pi)(u)=u^{N_i}$. Let $\chi=(l\circ \pi)^*dx$, in a neighborhood of $p$, the 1-form $\chi$ is given by:
    $$
    \chi=-N_iu^{N_i-1}du.
    $$
    Since $l \circ \pi$ is a local isomorphism on the complement of $D$ in $Y$, the 1-form $\chi$ does not vanish on this set. Therefore, the canonical divisor $K_\chi$ on the smooth complex curve $Y$ is represented by:
    $$
    K_\chi = \sum_{i=1}^r(N_i-1)p_i+l^*.
    $$
    Hence, $\hat{k}_i=\text{ord}_{p_i}(\chi)=N_i-1$.\\

    Since $\pi$ induces an isomorphism $Y \smallsetminus \pi^{-1}(0) \longrightarrow C \smallsetminus \{0\}$, we have the following equality:
    \begin{align*}
        \mu_1(\underline{C} \cap S(0,n)) & = \mu_1(\underline{C}\smallsetminus\{0\} \cap S(0,n)) \\
        & = \int_{\pi^{-1}(\underline{C}\smallsetminus\{0\} \cap S(0,n))} \LL^{-\text{ord}(\pi^*(\Omega_X^1))}\\
    \end{align*}
    Hence, we just need to compute:
    $$
        \int_{\{x \mid v(x^{N_i})=n\}} \LL^{-(N_i-1)v(x)}.
    $$
    If $N_i$ does not divide $n$, this integral is equal to zero. If $N_i$ divides $n$, then:
    \begin{align*}
        \int_{\{x \mid v(x^{N_i})=n\}} \LL^{-(N_i-1)v(x)} &= \LL^{-(N_i-1)\frac{n}{N_i}}\mu(\{x \mid v(x^{N_i})=n\})\\
        & = \LL^{-(N_i-1)\frac{n}{N_i}} \LL^{-\frac{n}{N_i}-1}(\LL-1) \\
        & = \LL^{-n}(\LL-1).
    \end{align*}
    We renormalize by the volume of the one-dimensional sphere which is $ \LL^{-n}(\LL-1)$, we do the mean value at infinity and sum all of the terms, we finally get that:
    $$
    \Theta_1^\text{mot}(C,0) = \sum_{i=1}^r \frac{1}{N_i}.
    $$
\end{proof}

\vspace{\baselineskip}

Competing interests: The author declare none.

\bibliographystyle{plain}
\bibliography{Local_density_surface}

\end{document}